\documentclass[numbers,webpdf,imanum]{ima-authoring-template}%
\usepackage{amsmath,amsfonts,amssymb,mathtools,eucal,bm}
\usepackage{stmaryrd}
\usepackage{latexsym}
\usepackage{epsfig,overpic}
\usepackage{xcolor}
\usepackage{cite}
\usepackage{booktabs}

\theoremstyle{thmstyletwo}%
\newtheorem{theorem}{Theorem}[section]
\newtheorem{lemma}[theorem]{Lemma}%
\newtheorem{remark}{Remark}[section]%
\newtheorem{assumption}{Assumption}[section]

\numberwithin{equation}{section}

\let\ep\varepsilon

\newcommand{\bC}{\mathbf C}

\newcommand{\bV}{\mathbf V}

\newcommand{\bn}{\mathbf n}

\newcommand{\bng}{{\mathbf n}_\Gamma}

\newcommand{\bw}{\mathbf w}
\newcommand{\be}{\mathbf e}
\newcommand{\bx}{\mathbf x}

\newcommand{\bz}{\mathbf z}

\newcommand{\bu}{\mathbf u}

\newcommand{\bv}{\mathbf v}
\newcommand{\blf}{\mathbf f}
\newcommand{\bH}{\mathbf H}

\newcommand{\R}{\mathbb R}

\newcommand{\bpsi}{{\bm \psi}}

\newcommand{\E}{\mathcal E}

\newcommand{\Q}{\mathcal Q}

\newcommand{\Div}{\operatorname{\rm div}}

\newcommand{\dO}{\partial\Omega}
\newcommand{\whO}{\widehat\Omega}
\newcommand{\calF}{\mathcal{F}}
\newcommand{\calS}{\mathcal{S}}
\newcommand{\calO}{\mathcal{O}}
\newcommand{\calQ}{\mathcal{Q}}
\newcommand{\p}{\partial}
\newcommand{\jump}[1]{\bigl[\hspace{-0.03in}\bigl[#1\bigr]\hspace{-0.03in}\bigr]}

\newcommand{\pol}{{\rm P}}
\newcommand{\bpol}{{\bf P}}
\newcommand{\calE}{\mathcal{E}}
\newcommand{\bt}{{\bf t}}
\newcommand{\tnorm}[1]{|\!|\!| #1 |\!|\!|}
\newcommand{\bbU}{\mathbb{U}}
\newcommand{\bbP}{\mathbb{P}}
\newcommand{\Thi}{\mathcal{T}_{h,i}^n}
\newcommand{\The}{\mathcal{T}_{h,e}^n}
\newcommand{\Fhi}{\mathcal{F}_{h,i}^n}
\newcommand{\Fhe}{\mathcal{F}_{h,e}^n}
\newcommand{\Omh}{\Omega_{h}^n}
\newcommand{\Omhe}{\Omega_{h,e}^n}
\newcommand{\Omhi}{\Omega_{h,i}^n}
\newcommand{\bW}{{\bf W}}
\newcommand{\qnorm}[1]{|\!|\!| #1 |\!|\!|_{n,i}}
\newcommand{\qnorme}[1]{|\!|\!| #1 |\!|\!|_{n,e}}
\newcommand{\vnorm}[1]{|\!|\!| #1 |\!|\!|_n}
\newcommand{\vnorme}[1]{|\!|\!| #1 |\!|\!|_{n,e}}

\newcommand{\bbR}{\mathbb{R}}

\renewcommand{\O}{\mathcal{O}}
\usepackage{accents}

\newcommand{\nab}{\nabla}
\newcommand{\calT}{\mathcal{T}}

\begin{document}

\copyrightyear{2023}
\vol{}
\firstpage{1}


\title[FEM for linearized Navier--Stokes problem in an evolving domain]{An Eulerian finite element method for the linearized Navier--Stokes problem in an evolving domain}

\author{Michael Neilan*
\address{\orgdiv{Department of Mathematics}, \orgname{University of Pittsburgh}, \orgaddress{\street{Pittsburgh}, \state{PA}, \postcode{15260}, \country{USA}}}}
\author{Maxim Olshanskii
\address{\orgdiv{Department of Mathematics}, \orgname{University of Houston}, \orgaddress{\street{Houston}, \state{TX}, \postcode{77204}, \country{USA}}}}

\newcommand{\rev}[1]{\textcolor{black}{#1}}

\authormark{M.~Neilan and M.~Olshanskii}

\corresp[*]{Corresponding author: \href{neilan@pitt.edu}{neilan@pitt.edu}}

 \received{2}{8}{2023}


\abstract{
The paper addresses an error analysis of an Eulerian finite element method used for solving a linearized Navier--Stokes problem in a time-dependent domain. In this study, the domain's evolution is assumed to be known and independent of the solution to the problem at hand. The numerical method employed in the study combines a standard Backward Differentiation Formula (BDF)-type time-stepping procedure with a geometrically unfitted finite element discretization technique. Additionally, Nitsche's method is utilized to enforce the boundary conditions.
 The paper presents a convergence estimate for several velocity--pressure elements that are inf-sup stable. The estimate demonstrates optimal order convergence in the energy norm for the velocity component and a scaled $L^2(H^1)$-type norm for the pressure component.}
\keywords{Interface Stokes problem; evolving interface; cutFEM.}


\maketitle

\section{Introduction} 
Fluid equations formulated in time-dependent domains are essential components of mathematical models used in a wide range of applications, including cardiovascular research and aerospace engineering~\citep{formaggia2010cardiovascular,bazilevs2013computational}. The analysis of such equations is a classical topic in mathematical fluid mechanics~\citep{solonnikov1977solvability,miyakawa1982existence,solonnikov1987transient,neustupa2009existence}. Moreover, a significant body of literature  addresses the development of computational methods aimed at numerically solving these problems.
Well-established numerical techniques include immersed boundary methods, fictitious domain methods, methods based on Lagrangian and arbitrary Lagrangian–Eulerian formulations, space–time finite element formulations, level-set methods, and extended finite element methods; see, e.g., \citep{hirt1974arbitrary,peskin1977numerical,tezduyar1992new,masud1997space,glowinski1999distributed,formaggia1999stability,duarte2004arbitrary,gross2011numerical}.

In this paper, we focus on an Eulerian finite element method that utilizes a time-independent triangulation of $\mathbb{R}^3$ to solve a system of governing equations within a volume $\Omega(t)$ that smoothly evolves through the background mesh, a typical configuration for unfitted finite element methods. Specifically, we consider the CutFEM unfitted finite element method~\citep{cutFEM} that incorporates Nitsche's method for boundary condition imposition and employs a ghost-penalty stabilization~\citep{B10} to handle instabilities arising from arbitrary small ``cuts" made by $\Omega(t)$ within the background simplices.

For time stepping, we adopt an Eulerian procedure suggested in~\citep{lehrenfeld2019eulerian} that relies on the implicit extension of the solution from $\Omega(t)$ to its neighborhood of $\mathcal{O}(\Delta t)$. This combination of the CutFEM method and implicit extension-based time stepping was initially applied to two-phase flow problems in\citep{claus2019cutfem}, demonstrating its efficacy when used in conjunction with the level-set method for interface capturing. Recent studies in \citep{burman2022eulerian} and \citep{von2022unfitted} have addressed the analysis of this method, considering equal-order stabilized and Taylor-Hood elements, respectively. Both of these analyses identified a major challenge: the lack of a weak divergence-free property of the time difference of the finite element solutions $(\mathbf{u}_h^n-\mathbf{u}_h^{n-1})/\Delta t$ with respect to the discrete pressure space at time $t^n$. The absence of this property makes it challenging to bound this term in a suitable norm and precluding optimal-order estimates for the pressure. This observation has also been made in the literature on adaptive-in-time finite element methods, where the pressure space varies in time due to mesh adaptation~\citep{besier2012pressure,brenner2014priori}. The use of equal-order finite elements and pressure stabilization in \citep{burman2022eulerian} allows the authors to establish the optimal error estimate for velocity. However, for inf-sup stable Taylor-Hood elements, the coupling between pressure and velocity appears stronger, and the sub-optimality in pressure also hindered the authors of \citep{von2022unfitted} from obtaining the optimal order estimation for the velocity error. It is worth noting that \citep{von2022unfitted} also quantified the error resulting from an approximate reconstruction of the evolving ``exact" domain, $\Omega(t)$.

Despite the aforementioned theoretical challenges, numerical experiments have demonstrated 
optimal order convergence rates~\citep{von2022unfitted}. This raises the question of whether the analysis can be enhanced to provide support for the observed numerical evidence. This is the question addressed in the present paper. The setup of the problem and the methods here is similar to~\citep{von2022unfitted}, but we consider general inf-sup stable unfitted finite element pairs, essentially those covered in the analysis by Guzmán et al.~\citep{guzman2016inf}.

The main result established in this paper can be summarized as follows: Optimal convergence rates are proven for the energy norm of velocity and a scaled $L^2(H^1)$-norm of the pressure under the constraint $h^2\lesssim \Delta t\lesssim h$, where $h$ represents the mesh size and $\Delta t$ denotes the time step. This  bridges the gap in the analysis up to the selection of the pressure norm. Notably, the use of a non-standard pressure norm is vital in mitigating the lack of divergence-free property in the discrete time derivative. This argument aligns with the analysis in a recent study \citep{olshanskii2023eulerian}, which analyzed a finite element method for the Navier-Stokes equations posed on time-dependent surfaces.

In general, there is a scarcity of literature addressing error bounds for fully discrete solutions of fluid equations in evolving domains. However, under the simplifying assumption that the motion of the domain is given and decoupled from the flow solution, error bounds for the Arbitrary Lagrangian-Eulerian (ALE) and quasi-Lagrangian finite element methods for Stokes, Navier-Stokes and coupled Stokes--parabolic equations in moving domains can be found in~\citep{san2009convergence,lozovskiy2018quasi,KeslerLanSun21}. Similarly, error bounds for the unfitted characteristic finite element method within the same setup are provided in~\citep{ma2023high}. 

The remainder of the paper is organized in five sections and an appendix. Section~\ref{s:form} formulates the linearized Navier-Stokes problem in evolving domains
and introduces suitable extension operators utilized in the analysis.
In particular, the numerical analysis relies on existence of a sufficiently regular  divergence-free extension of the fluid velocity field in a neighborhood of $\Omega(t)$.
The fully discrete numerical method based on a Nitsche-based CutFEM formulation is given in Section~\ref{s:FEM}. 
Here, we present the scheme for general finite element Stokes pairs satisfying certain assumptions.
Stability and convergence analysis is the subject of Section~\ref{s:NA}.
In Section \ref{sec:Examples}, we list three standard finite element pairs
satisfying the assumptions. Finally, a proof of
a `discrete' trace estimate is found in Appendix \ref{app:Dtrace}.

\section{Problem formulation} \label{s:form}

We consider a time-dependent domain $\Omega(t)\subset \bbR^3$
with boundary $\Gamma(t):=\p \Omega(t)$ whose motion is assumed to
be known a priori. In particular, we assume
a smooth solenoidal vector field $\bw:\bbR^3\times [0,T]\to \bbR^3$, 
for some final time $T>0$
such that the normal velocity of the boundary is specified via
\begin{equation}\label{VG}
V_\Gamma=\bw\cdot\bng\quad \text{on}~\Gamma(t),
\end{equation}
 where $\bn_{\Gamma}$ denotes the outward unit normal 
 of $\Gamma(t)$.
We then consider the Oseen problem in the moving volume $\Omega(t)$:
	\begin{equation} \label{eqn:Stokes}
		\begin{split}
			\bu_t+(\bw\cdot\nabla)\bu- \Delta \bu +\nabla p &= \blf~\quad  \text{in }\Omega(t),\\
			\Div \bu & = 0~\quad \text{in }\Omega(t),\\
		\bu&=\bw\quad  \text{on }\Gamma(t),
		\end{split}
	\end{equation}
 with initial condition $\bu|_{t=0}=\bu_0$ in  $\Omega_0:=\Omega(0)$. As mentioned in the introduction, unfitted finite element methods for \eqref{eqn:Stokes} were recently addressed in~\citep{burman2022eulerian,von2022unfitted} with suboptimal error bounds.  We note that the previous studies \citep{burman2022eulerian,von2022unfitted} ignore the advection term $(\bw\cdot\nabla)\bu$ in \eqref{eqn:Stokes}.
 While this term does not lead to any additional difficulties in the analysis, we believe it is 
 mechanically relevant to include it in this simplified model.
 By a standard argument, we can re-write the above problem for  
 \begin{equation}\label{bc}
 \bu=0\quad \text{on} ~~\Gamma(t).
 \end{equation}
 
We assume the smooth velocity field $\bw\,:\,\mathbb{R}^3\times [0,T]\to\mathbb{R}^3$ is such that
it  defines the flow map $\Phi_t:\, \Omega(0) \to \Omega(t)$ as the material evolution of the fluid volume:  For $\bz \in \Omega_0$, the trajectory $\bx(t,\bz)=\Phi_t(\bz)$  solves 
 \begin{align}
 	\begin{cases}\label{def_flow_map}
 			\bx(0,\bz)=\bz,\\
 			\frac{d}{dt} \bx(t,\bz)= \bw(t,\bx(\bz,t))\qquad t\in (0,T]
 		\end{cases}
 \end{align}
for some final time $T>0$.
Equation \eqref{def_flow_map} defines a 
smooth bijection between $\Omega_0$ and $\Omega(t)$ for 
every $t\in[0,T]$. If $\dO_0\in C^p$ and $\bw\in \bC^p(\mathbb{R}^3)$, then $\Gamma(t)\in C^p$; the flow map $\Phi_t$ also preserves the connectivity of $\Omega(t)$. 

Summarizing,  we are interested in the analysis of a finite element method for  solving \eqref{eqn:Stokes} with $\Omega(t)=\Phi_t(\Omega(0))$ and homogeneous Dirichlet boundary conditions~\eqref{bc}.

\subsection{Extensions} \label{s:discretization}
Let $\Omega(t)\subset\whO$ for all $t\in [0,T]$, for a bounded polyhedral domain $\whO\subset\mathbb{R}^3$. 
We define the two space-time domains $\Q$ and $\hat \Q$ 
as follows:
\[
\Q:=\bigcup\limits_{t\in[0,T]}\Omega(t)\times\{t\}\subset \widehat\Q:=\whO\times [0,T]\subset \mathbb{R}^4.
\]
For a 
domain $D\subset \R^3$ and some $\delta>0$ we use the notation $\calO_\delta(D)$ for
the $\delta$-neighborhood of $D$:
\[
\calO_\delta(D) = \{x\in \mathbb{R}^3:\ {\rm dist}(x,D)\le \delta\}.
\]

Denoting by $\bV(t) = \{\bv \in \bH^1_0(\Omega(t)):\ {\rm div}\,\bv=0\}$, the subspace of divergence-free functions in $\bH^1_0(\Omega(t))$,
our goal now is to  define an extension operator $\E\,:\,\bV(t)\to \bH^1(\whO)$ that preserves the divergence--free condition.
To this end, we note that since  $\Div\bu=0$, we can write  $\bu = \nab \times \bpsi$ in $\Omega(t)$ with a stream function that satisfies $\bpsi\in \bW^{k+1,p}(\Omega(t))$
and
\begin{equation}\label{aux318}
\|\bpsi\|_{W^{k+1,p}(\Omega(t))}\lesssim  \|\bu\|_{W^{k,p}(\Omega(t))}\quad\text{for}~\bu\in W^{k,p}(\Omega(t)),
\end{equation}
$k\ge0$, $1<p<\infty$; see~\citep{GR,costabel2010bogovskiui}. 
\begin{remark}\rm 
Here, the statement $A \lesssim B$ (resp., $A\gtrsim B$) to mean $A\le c  B$ (resp., $A\ge c B$) for some constant $c>0$ independent of 
    the spatial and temporal discretization parameters $h$ and $\Delta t$ introduced below and time $t$.
    The statement $A\simeq B$ means $A\lesssim B$ and $A\gtrsim B$.
\end{remark}

For $\bpsi_0=\bpsi \circ \Phi_t$ we consider Stein's extension:  Since the boundary of $\Omega_0$ is smooth, there is a continuous linear extension operator $\E_{0} : L^2(\Omega_{0}) \to  L^2(\R^3)$,  ($\E_{0} \bpsi_0 = \bpsi_0$ in $\Omega_{0}$), with the following  properties~\citep[Section~VI.3.1]{stein2016singular}:
\begin{equation}\label{ExtBound0}
  \Vert {\E_{0}} \bpsi_0 \Vert_{W^{k,p}(\R^3)} \leq  C_{\Omega_{0}} \Vert \bpsi_0 \Vert_{W^{k,p}(\Omega_{0})},\quad \text{for}~\bpsi_0 \in W^{k,p}(\Omega_{0}),~~k=0,\dots,m+1,~~1\le p\le\infty,
\end{equation}
with any fixed $m\ge0$.  Here, the extension operator is performed
component-wise, i.e., $(\E_0 \bpsi_0)_i = \E_0 (\bpsi_0)_i$ for $i=1,2,3$.
For the extension $\E_{\psi} \bpsi := (\E_{0} \bpsi_0 ) \circ \Phi^{-1}_t$ of $\bpsi$,
the following estimates follow from the analysis in  \citep{lehrenfeld2019eulerian}: 
\begin{equation}\label{aux329}
\begin{aligned}
\|\E_{\psi} \bpsi \|_{H^{k}(\whO)}&\lesssim \| \bpsi\|_{H^{k}(\Omega(t))}, \quad {\small k=0,\dots,m+1},\quad
		\|\E_{\psi} \bpsi \|_{W^{4,5}(\widehat \Q)}\lesssim \|\bpsi\|_{W^{4,5}(\Q)}, \\
  	\|(\E_{\psi} \bpsi)_t\|_{H^{m}(\whO)}&\lesssim (\| \bpsi\|_{H^{m+1}(\Omega(t))}+\|\bpsi_t\|_{H^{m}(\Omega(t))}).
\end{aligned}
\end{equation}
We now define the velocity extension as follows
\begin{equation} \label{e:extensiont_cont}
\E{} \bu(t) := \nab \times (\E_{\psi} \bpsi),\quad \text{for each}~t \in [0,T].
\end{equation}
By construction there holds
\[
 \Div\E\bu=0\quad \text{in}\, \whO.
\]
For $\bu\in L^\infty(0,T;\bH^{m}(\Omega(t)))\cap \bW^{3,5}(\Q)$ such 
that $\Div\bu=0$ in $\Omega(t)$ for all $t\in(0,T)$ and any fixed integer $m\ge0$, the following estimates follow from \eqref{aux318}, \eqref{aux329}, Poincare-Friedrich's inequality, 
and the embedding $\ W^{3,5}(\widehat \Q)\subset W^{2,\infty}(\widehat \Q)$:
\begin{subequations}\label{u_bound_a}
	\begin{align}
		\|\E\bu\|_{H^{k}(\whO)}&\lesssim \|\bu\|_{H^{k}(\Omega(t))}, \quad {\small k=0,\dots,m},\label{u_bound_a1}\\
		\Vert \nabla (\E\bu) \Vert_{\whO} & \lesssim \Vert \nabla \bu \Vert_{\Omega(t)},\label{u_bound_a2}\\
		\|\E\bu\|_{W^{2,\infty}(\widehat \Q)}&\lesssim \|\bu\|_{W^{3,5}(\Q)}, \label{u_bound_a3}
	\end{align}
\end{subequations}
Here, we use the standard notation for the $L^2$-norm $\|\cdot\|_D = \|\cdot\|_{L^2(D)}$ for some domain $D$.
Furthermore, for
$\bu\in L^\infty(0,T;\bH^{m}(\Omega(t)))$ such that $\bu_t\in L^\infty(0,T;\bH^{m-1}(\Omega(t)))$ it holds
\begin{equation}\label{u_bound_b}
	\|(\E\bu)_t\|_{H^{m-1}(\whO)}\lesssim (\|\bu\|_{H^{m}(\Omega(t))}+\|\bu_t\|_{H^{m-1}(\Omega(t))}).
\end{equation}
With an abuse of notation, we define the extension
of the pressure as
\begin{align} \label{e:extensiont_cont_press}
 \E p(t) = (\E_0(p\circ \Phi_t))\circ \Phi_t^{-1},\quad \text{for each}~t\in [0,T].   
\end{align}
Then estimates \eqref{u_bound_a1},\eqref{u_bound_a3}, with $\calE \bu$ and $\bu$ replaced by
$\calE p$ and $p$, respectively, are satisfied (cf.~\citep[Lemma 3.3]{lehrenfeld2019eulerian}).
For the analysis, we need $\E\bu$ and $\E p$ defined in $\O_\delta(\Omega(t))\subset\whO$, a $\delta$-neighborhood of $\Omega(t)$ with $\delta\simeq\Delta t$.

\section{The Fully Discrete Finite Element Method} \label{s:FEM}
We adopt the basic framework
in \citep{lehrenfeld2019eulerian,von2022unfitted,burman2022eulerian} 
to build a Nitsche-based CutFEM spatial discretization
of the Stokes problem on an evolving domain.

\subsection{Approximate geometries}
Recall that $\whO\subset\mathbb{R}^3$ is a polyhedral domain
with $\Omega(t)\subset \whO$ for all $t\in [0,T]$.
For simplicity, we consider a time discretization
with a uniform time-step $\Delta t = T/N$ for some $N\in \mathbb{N}$.
We set $t_n = n \Delta t$, $\Omega^n = \Omega(t_n)$, $\Gamma^n = \Gamma(t_n)$, 
and $(\bu^n,p^n) = (\bu(t_n),p(t_n))$.
We further set $\bw^n_\infty = \|\bw(t_n)\cdot \bn_{\Gamma}\|_{L^\infty(\Gamma^n)}$.
For practical purposes such as numerical integration, and similar to \citep{lehrenfeld2019eulerian,von2022unfitted,burman2022eulerian}, 
we assume that the domains $\Omega^n$ 
are given by their
 approximations
$\Omh$ (cf.~\eqref{eqn:PhiIsGood}--\eqref{eqn:PhiIsGood2} below).
The boundary of $\Omh$ is denoted by $\Gamma_h^n$.
To ensure that discrete solutions are well defined
at subsequent time-steps, we extend 
the computational domain by a layer of thickness $\delta_h$
with $c_{\delta_h} \bw^n_\infty \Delta t\le  \delta_h$
with constant $1\le c_{\delta_h}=O(1)$ such that
 ${\rm dist}(\Omega_h^n,\Omega_h^{n+1})\le \delta_h$
for all $n$.

We assume there is a bijective, Lipschitz continuous 
map $\Psi_n: \calO_{\delta_h}(\Omh)\to \calO_{\delta_h}(\Omega^n)$
that connects the approximate and exact domains at each time step.
In particular, we assume  $\Psi_n$ satisfies
 $\calO_{\delta_h}(\Omega^n) = \Psi_n(\calO_{\delta_h}(\Omh))$,
 $\Omega^n = \Psi_n(\Omh)$, $\Gamma^n = \Psi_n(\Gamma_h^n)$, 
 and the existence
 of a positive integer $q$ such that
\begin{align}\label{eqn:PhiIsGood}
    \|\Psi_n - {\rm id}\|_{W^{j,\infty}(\calO_{\delta_h}(\Omh))} \lesssim h^{q+1-j}\qquad  j=0,1.
\end{align}
We refer to $q$ as the geometric order of approximation. 
Such a mapping has been constructed in \citep{gross2015trace}
based on isoprametric mappings of geometries defined via
level sets.  Note that \eqref{eqn:PhiIsGood} implies
\begin{equation}\label{eqn:PhiIsGood2}
{\rm dist}(\Omega^n,\Omh)\lesssim h^{q+1}.
\end{equation}

\subsection{Triangulations}

We let $\calT_h$ denote a shape-regular and quasi-uniform
simplicial triangulation of the background domain $\whO$
with $h = \max_{T\in \calT_h} {\rm diam}(T)$.
Note the quasi-uniformity implies a constant $c>0$
such that $h\le c\, {\rm diam}(T)=:h_T$ for all $T\in \calT_h$.

We then define, for each time step $n$, the active triangulation and corresponding domain induced
by the background triangulation (cf.~Figure \ref{fig:Tri}): 
\begin{align*}
\The = \{T\in \calT_h:\ {\rm dist}(\bx,\Omega_h^n)\le \delta_h\ \exists\, \bx\in \bar T\},\qquad \Omhe = {\rm int}\left(\bigcup_{T\in \The} \bar T\right).
\end{align*}
We further definite the set of interior elements for $\Omh$ and associated domain
at time step $n$:
\[
\Thi = \{T\in \The:\ {\rm int}(T)\subset \Omh\},\qquad \Omhi = {\rm int}\left(\bigcup_{T\in \Thi} \bar T\right),
\]
and denote by $\Fhi$ (resp., $\Fhe$) the set of interior faces of $\Thi$ (resp., $\The$), i.e.,
\[
\calF_{h,*}^n = \{F = \p T_1\cap \p T_2:\ T_1,T_2\in \calT_{h,*}^n,\ T_1\neq T_2\}\qquad *\in \{i,e\}.
\]
We further set $h_F = {\rm diam}(F)$ for all $F\in \calF_{h,e}$.
Following \citep{lehrenfeld2019eulerian,von2022unfitted},
we define the elements in a strip around $\Gamma_h^n$:
\[
\calT_{\Gamma_h}^n:= \{T\in \The:\ {\rm dist}(x,\Gamma_h^n)\le \delta_h\ \exists\, x\in \bar T\},
\]
and define the set of faces in this strip:
\[
\calF_{\Gamma_h}^n:= \{F=\p T_1\cap \p T_2:\ T_1\in \calT_{h,e}^n,\ T_2\in \calT_{\Gamma_h}^n,\ T_1\neq T_2,\ |\p T_1\cap \p T_2|>0\}.
\]

For any sub-triangulation $\calS_h \subset \calT_h$ and $m\in \mathbb{N}$, 
we set $H^m(\calS_h)$ to be the piecewise Sobolev space
with respect to $\calS_h$, i.e., $q\in H^m(\calS_h)$ implies
$q$ is an $L^2$ function on the domain induced by $\calS_h$
and $q|_T\in H^m(T)$ for all $T\in \calS_h$.
Analogous vector-valued spaces are denoted in boldface.

\begin{figure}[h]
\centering
\includegraphics[scale=0.5]{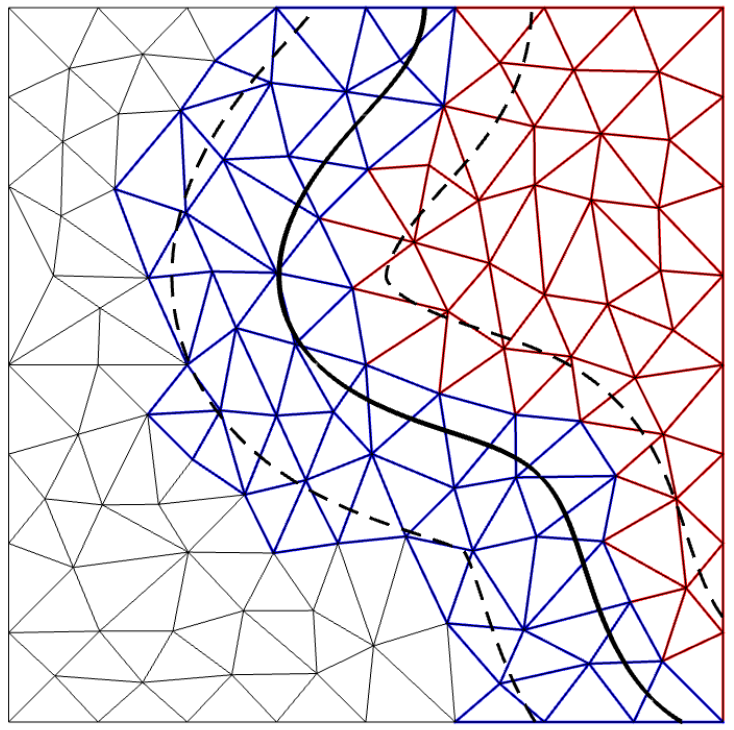}\qquad
\includegraphics[scale=0.5]{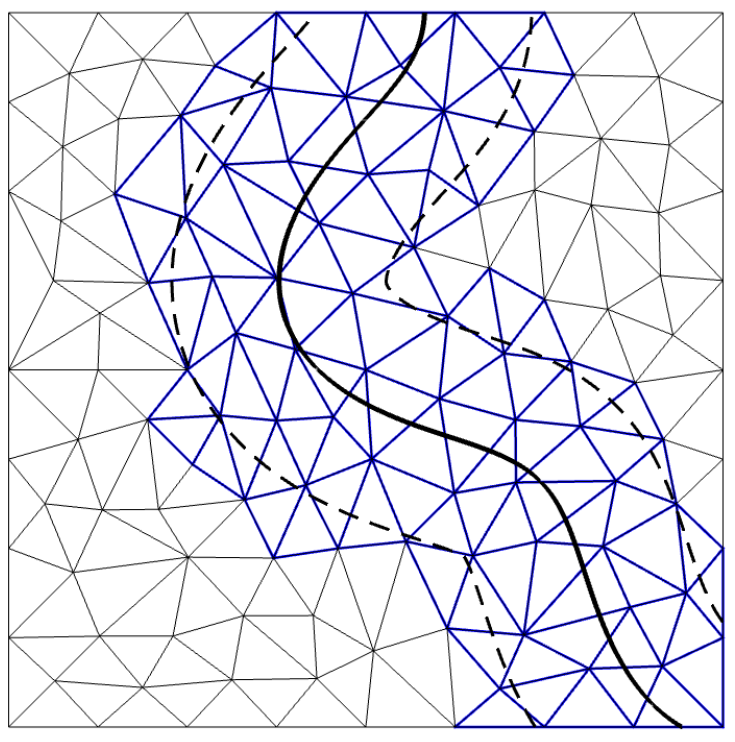}
\caption{\label{fig:Tri}Left: A depiction of the mesh
in two dimensions,
with the interior triangulation $\calT_{h,i}^n$
in red and $\calT_{h,e}^n\backslash \calT_{h,i}^n$ in blue.
Right: The triangulation $\calT_{\Gamma_h}^n$ (blue) around
a $\delta_h$-neighborhood tubular region of $\Gamma_h$.
}
\end{figure}

\subsection{Finite Element Spaces and Assumptions}
We denote by $\pol_m(\calT_h)$ the space 
of piecewise polynomials of degree $m$ with 
respect to $\calT_h$, and set $\pol_m^c(\calT_h) = \pol_m(\calT_h)\cap H^1(\whO)$
to be its subspace of continuous, piecewise polynomials. Analogous
vector-valued spaces are denoted in boldface.
We consider a Stokes finite element pair
$\bV_h\times Q_h\subset \bH^1(\whO)\times L^2(\whO)$,
consisting of piecewise polynomial spaces with respect to $\calT_h$,
and assume the following inclusions
\begin{align}\label{eqn:VhInclusions}
    \bpol_{\underline{m}_v}^c(\calT_h)&\subset \bV_h\subset \bpol_{\overline{m}_v}^c(\calT_h),
\end{align}
for some integers $1\le \underline{m}_v\le \overline{m}_v$. We further assume 
there exists $m_q\in \mathbb{N}_0$ such that 
\begin{align}\label{eqn:QhInclusions}
 Q_h = \pol_{m_q}(\calT_h)\quad \text{or}\quad  Q_h = \pol_{m_q}^c(\calT_h).
\end{align}

We set $\bV_h^n \subset \bH^1(\Omhe)$ to
be the restriction of $\bV_h$
to $\Omhe$, and let $Q_h^n$ 
be the restriction of $Q_h$ to $\Omhe$ with a zero-mean constraint 
on $\Omega_{h,i}^n$, i.e.,
\[
Q_h^n = \{q|_{\Omhe}:\ \exists\, q\in Q_h\text{ such that }\int_{\Omega_{h,i}^n}q\,dx=0 \}.
\]
Note that, by construction, $\Omh \subset \Omega_{h,e}^{n-1}$,
and therefore functions in $\bV_h^{n-1}\times Q_h^{n-1}$ are well defined on $\Omh$.

We define the Nitsche-type norms on $\bH^1(\Omh)\cap \bH^2(\The)\big|_{\Omh}$:
\[
\vnorm{\bv}^2:=\|\nab \bv\|_{\Omega_h^n}^2 + h^{-1} \|\bv\|_{\Gamma^n_h}^2 + h\|\nab \bv\|_{\Gamma_h^n}^2,\qquad
\]
and further define the norm 
for piecewise smooth functions on the extended domains:
\[
\vnorme{\bv}^2:= \|\nab \bv\|_{\Omhe}^2+ \tnorm{\bv}_n^2.\qquad
\]
Likewise, we define 
the weighted $H^1$-seminorm 
with respect to the interior mesh $\Thi$:
\[
\qnorm{q}^2:=\rev{h^2}\sum_{T\in \Thi} \|\nab q\|_T^2+\rev{h}\sum_{F\in \Fhi} \left\|\jump{q}\right\|_F^2,
\]
where $\jump{\cdot}$ denotes the jump operator across an interior face.
Note that $\qnorm{\cdot}$ is a norm on $Q_h^n|_{\Omega_{h,i}^n}$.
Similarly, 
we define weighted seminorm over the extended domain  $\Omhe$:
\[
\qnorme{q}^2:=\rev{h^2}\sum_{T\in \The} \|\nab q\|_T^2+\rev{h}\sum_{F\in \Fhe} \left\|\jump{q}\right\|_F^2.
\]
Note that $\qnorme{q}$ is a norm on $Q_h^n$, and \emph{it will be our main pressure norm for stability and error analysis.} 

In addition to the inclusions \eqref{eqn:VhInclusions}--\eqref{eqn:QhInclusions} ,
we make the following assumptions to ensure stability of the discretization presented below.
\begin{assumption}\label{ass:Pair}
Assume that, given $q\in Q_h^n$,
there exists $\bv_h\in \bV_h^n$ that satisfies
 \begin{subequations}
\label{eqn:DinfSupExtra}
     \begin{align}
     \label{eqn:DinfSupExtra1}
         \vnorme{\bv}&\lesssim \qnorm{q},\\
         \label{eqn:DinfSupExtra2}
         \qnorm{q}^2&\le b^n_h(\bv,q):=\int_{\Omh} ({\rm div}\,\bv)q\, dx -\int_{\Gamma_h^n} (\bv\cdot \bn) q\, ds,\\
         \label{eqn:DinfSupExtra3}
         \|\bv\|_{\Omh}&\lesssim h\qnorm{q}.
     \end{align}
 \end{subequations}
 \end{assumption}

\begin{remark}\label{rem:InfSup} \rm
The first two statements \eqref{eqn:DinfSupExtra1}--\eqref{eqn:DinfSupExtra2} are
assumptions related to discrete inf-sup stability, but where the $L^2$ norm
of the pressure function is replaced with the weighted $H^1$-norm.  A variation
of these conditions is shown to hold in the context of CutFEM
for many standard stable Stokes pairs in \citep{guzman2016inf}.
Using a Verf\"urth-type trick, it is shown in this reference that,
if \eqref{eqn:DinfSupExtra1}--\eqref{eqn:DinfSupExtra2} is satisfied
then the discrete inf-sup condition with $L^2$ pressure norm holds:
\[
\theta \|q\|_{\Omh} \le \sup_{\bv\in \bV_h^n} \frac{b_h^n(\bv,q)}{\vnorme{\bv}}+|q|_{J^n_h}\qquad \forall q\in {Q}_h^n,
\]
where $\theta>0$ is independent of $h$ and how $\Gamma_h^n$ cuts through the triangulation $\calT_h$,
 and $|\cdot|_{J_h^n}$ is given by \eqref{eqn:GhostsemiNorms} below.
We show below in Section \ref{sec:Examples} that the third condition \eqref{eqn:DinfSupExtra3} is satisfied for several canonical pairs as well.    
\end{remark}

\begin{remark}\label{rem:AssumeGeneral}\rm
Assumption \ref{ass:Pair} can be modified
and slightly weakened by replacing $\Omhi$
and $\Thi$ by a smaller domain and mesh, respectively,
provided the pressure ghost-penalty compensates for the smaller domain.
In particular, let $\tilde \calT_{h,i}^n\subset \Thi$ be a sub-mesh
with corresponding domain $\tilde \Omega_{h,i}^{n} = {\rm int}\left(\bigcup_{T\in \tilde \calT_{h,i}^n} \bar T\right)$.
Then if 
 \begin{equation}\label{eqn:qExtend}
 \|q\|_{\Omhe} \lesssim \|q\|_{\tilde \Omega_{h,i}^{n}} + |q|_{J^n_h}\qquad \forall q\in Q_h^n,
 \end{equation}
then we can replace $\calT_{h,i}^n$ by $\tilde \calT_{h,i}^n$ and $\Omega_{h,i}^{n}$ by $\tilde \Omega_{h,i}^{n}$
in Assumption \ref{ass:Pair}. This modified assumption is used in the case $\bV_h\times Q_h$ is the Taylor-Hood pair.
\end{remark}

\subsection{The CutFEM Discretization}
The finite element method based on the backward Euler temporal discretization seeks, at each time step, the pair $(\bu_h^n,p_h^n)\in \bV_h^n\times Q_h^n$
such that
\begin{equation}\label{eqn:FEM}
    \int_{\Omega_h^n} \Big(\frac{\bu_h^n-\bu_h^{n-1}}{\Delta t}\Big)\cdot \bv\,dx+ a^n_h(\bu_h^n,\bv) -b_h^n(\bv,p_h^n)
    + b_h^n(\bu_h^n,q) +\gamma_J J_h^n(p_h^n,q) = F^n(\bv,q),
\end{equation}
for all $\bv\in \bV_h^n,\, q\in Q_h^n$.
Here, $b_h^n(\cdot,\cdot)$ is given by \eqref{eqn:DinfSupExtra2},
and the bilinear form
$a_h^n(\cdot,\cdot)$ is defined as
\begin{align*}
\widehat a_h^n(\bu,\bv) & = \int_{\Omega_h^n} \nab \bu:\nab \bv\, dx + \int_{\Omega_h^n} (\bw \cdot \nab \bu)  \cdot \bv\,dx - \int_{\Gamma_h^n} \left([(\nab \bu) \bn]\cdot \bv+ [(\nab \bv) \bn]\cdot \bu - \frac{\eta}{h} \bu\cdot \bv\right)\, ds,\\
a_h^n(\bu,\bv) & =\widehat a_h^n(\bu,\bv)+\gamma_s s^n_h(\bu,\bv),
\end{align*}
where 
$\gamma_s,\gamma_J,\eta \ge 1$ are user-defined constants.
The bilinear forms
$s_h^n(\cdot,\cdot)$ and $J_h^n(\cdot,\cdot)$ consist 
of ghost-penalty terms
acting on $\bV^n_h\times \bV_h^n$ and $Q_h^n\times Q_h^n$, respectively, 
defined on an $O(\delta_h)$ neighborhood of $\Gamma_h^n$:
\begin{equation}\label{eqn:GPForms}
\begin{split}
s_h^n(\bu,\bv) 
&= \sum_{F\in \calF_{\Gamma_h}^n} \sum_{k=1}^{\overline{m}_v} h^{2k-1}\int_F \jump{\p^k_{\bn_F} \bu} \jump{\p^k_{\bn_F} \bv}\, ds,\\
J_h^n(p,q) &= \sum_{F\in \calF_{\Gamma_h}^n} \sum_{k=0}^{m_q} h^{2k+1} \int_F  \jump{\p^k_{\bn_F} p}\jump{\p^k_{\bn_F}q}\, ds,
\end{split}
\end{equation}
and $\p^k_{\bn_F}$ denotes the $k$th-order directional derivative 
with respect to the normal of the face $F$.
Here, $\overline{m}_v$ and $m_q$ are the integers in \eqref{eqn:VhInclusions}--\eqref{eqn:QhInclusions}.
Finally, $F^n(\bv,q)$ is a bounded linear functional on $\bV^n_h\times Q_h^n$ 
with 
\begin{equation}\label{eqn:Fnorm}
\|F^n\|_{*}:=\sup_{(\bv,q)\in \bV_h^n\times Q_h^n} \frac{F^n(\bv,q)}{(\vnorme{\bv}^2+ \qnorme{q}^2)^{\frac12}}<\infty.
\end{equation}
In \eqref{eqn:FEM} it is given by 
\[
F^n(\bv,q)=\int_{\Omega^n_h} {\bf f}^n\cdot \bv\, dx,
\]
but later we will consider a more general $F^n$ for the purpose of analysis.

\rev{
\begin{remark}\rm
We have assumed that $\bw$ is a (given) smooth solenoidal vector field defined on $\mathbb{R}^3\times [0,T]$.
 If this vector field is instead defined on $\cup_{t\in [0,T]} \Omega(t)\times \{t\}$, then a suitable extension to $\bw$ would be used in 
 the bilinear form $a_h^n(\cdot,\cdot)$.  
 Such an approximation may not be solenoidal, in which case
 a standard skew-symmetry of the convective term would be required
 in the finite element method.  
 The stability and convergence
 analysis results presented below still hold in this more general setting,
 albeit with slightly more technical arguments. We refer to \citep{gross2015trace}
 for details.
\end{remark}
}

\begin{remark}\rm
The ghost-penalty bilinear forms \eqref{eqn:GPForms}
both stabilize the solution of problem \eqref{eqn:FEM} due to irregular cuts
as well as yield implicit extensions to $\Omega_{h,e}^n$.
These terms also aid in algebraic stabilization, as the resulting
condition number of the system is insensitive to how $\Gamma_h^n$ intersects
$\calT_h$.
The pressure ghost-stabilization form $J_h^n(\cdot,\cdot)$
ensures numerical stability as it provides an inf-sup-type stability
condition of the pair $\bV_h^n\times Q_h^n$ (cf.~Remark \ref{rem:InfSup}).

There are now several types
of ghost-penalty stabilization besides the ``derivative jump version''
used in \eqref{eqn:GPForms}. These include the ``direct version'' \citep{preussmaster}
as well as the ``local projection stabilization version'' \citep{B10}.
In principle, we can replace \eqref{eqn:GPForms} with any choice of these
types of ghost penalty versions, and the stability and convergence analysis
presented below carries through with only superficial modifications.
However, for clarity of presentation, we only focus
on the derivative jump version in detail below.

Finally, we remark that the extension
of the discrete pressure approximation to all of 
 $\Omega_{h,e}^n$ is not required, in particular,
 the pressure ghost-penalty stabilization $J_h^n(\cdot,\cdot)$
 only needs to be defined on a single layer of elements around 
 $\Gamma_h^n$ to ensure stability. However, we use 
 the set of faces $\calF_{\Gamma_h}^n$ for both terms in 
 \eqref{eqn:GPForms} to simplify the presentation.
\end{remark}

\section{Stability and Convergence Analysis}\label{s:NA}
We denote by 
\begin{equation}\label{eqn:GhostsemiNorms}
|\bv|_{s_h^n} = \sqrt{s_h^n(\bv,\bv)}\quad\text{and}\quad
|q|_{J^n_h} = \sqrt{J_h^n(q,q)}
\end{equation}
the semi-norms induced by the bilinear forms 
$s_h^n(\cdot,\cdot)$
and $J_h^n(\cdot,\cdot)$, respectively.
We assume that the Nitsche penalty parameter $\eta$ is chosen sufficiently
large (but independent of $h$ and the mesh-interface cut) such that $a_h^n(\cdot,\cdot)$ is coercive on $\bV_h^n$ (cf.~\citep{burman2012fictitious}).
In particular, we assume $\eta>0$ is chosen such that
\begin{equation}\label{eqn:EtaCond}
a_h^n(\bv,\bv)\ge \frac12 \vnorm{\bv}^2 + \gamma_s |\bv|_{s_h^n}^2
\qquad \forall \bv\in \bV_h^n.
\end{equation}

Similar to \citep{lehrenfeld2019eulerian,von2022unfitted}, 
we assume that elements
in the strip $\The\backslash \calT_{h,i}^n$
can be reached from an uncut element in $\calT_{h,i}^n$
by a path that crosses at most $L$ faces with $L\lesssim (1+\frac{\delta_h}{h})$;
we refer to \citep{lehrenfeld2019eulerian,von2022unfitted} to see why this is
a reasonable assumption and how it relates to the shape-regularity of the 
triangulation $\calT_h$.
We consider the setting where $L$ is uniformly bounded with respect
to the discretization parameters, i.e., when $\delta_h\lesssim h$.
Recalling that $c_{\delta_h} \bw_\infty^n \Delta t\le \delta_h$
with $1\le c_{\delta_h} =O(1)$, this
brings us to the time-step restriction: 
\begin{equation}\label{cond1}
    \Delta t \lesssim h.
\end{equation}
The condition \eqref{cond1} and $\|\bw\|_{L^\infty(\Q)}\lesssim1$ implies
\begin{equation}\label{condL}
    L \lesssim 1.
\end{equation}
Thanks to \eqref{condL} and standard properties 
of the stabilization terms (see, e.g., \citep[Lemma 5.1]{massing2014stabilized}), 
we have the following norm equivalences for all $\bv\in \bV_h^n$ and $q\in Q_h^n$:
\begin{equation}\label{eqn:normEquivy}
\begin{split}
\|\bv\|_{\Omega_{h,e}^n}^2 &\simeq \|\bv\|_{\Omega_h^n}^2 + h^2  |\bv|_{s_h^n}^2,\\
\vnorme{\bv}^2 &\simeq  \vnorm{\bv}^2 +   |\bv|_{s_h^n}^2,  \\
\qnorme{q}^2 &\simeq  \qnorm{q}^2 + |q|_{J_h^n}^2,\\
\|q\|_{\Omhe}^2&\simeq \|q\|_{\Omega_{h,i}^n}^2 + |q|_{J_h^n}^2.
\end{split}
\end{equation}

\subsection{Preliminary Results}
In this section, we 
collect some preliminary results used in the stability
and the convergence analysis of the finite element method \eqref{eqn:FEM}.
\begin{lemma}
    For $h$ sufficiently small, there holds for all $\bv\in \bV_h^{n-1}$,
    \begin{equation}\label{eqn:ExtBound}
    \|\bv\|_{\Omega_h^n}^2\le \|\bv\|_{\Omega_{h,e}^{n-1}}^2\le (1+ c_1 \Delta t)\|\bv\|^2_{\Omega_h^{n-1}} + \frac{\Delta t}{4} \tnorm{\bv}_{n-1}^2 + \Delta t L |\bv|_{s_h^{n-1}}^2
    \end{equation}
    for a  constant $c_1>0$ independent of $h$, $\Delta t$
    and how the boundary cuts through the triangulation.
\end{lemma}
\begin{proof}
From \citep[Lemma 5.7]{lehrenfeld2019eulerian}, we have
\begin{align*}
\|\bv\|_{\Omega_{h,e}^{n-1}}^2
&\le (1+c_1(\epsilon)\Delta t) \|\bv\|_{\Omega_h^{n-1}}^2 + c_2(\epsilon) \Delta t \|\nab\bv\|_{\Omega_h^{n-1}}^2 +  c_3(\epsilon,h) \Delta t
L |\bv|_{s_h^{n-1}}^2,
\end{align*}
with 
\begin{alignat*}{2}
&c_1(\epsilon) = c' c_{\delta_h} \bw_\infty^n(1+\epsilon^{-1}),\qquad 
&&c_2(\epsilon) = c' c_{\delta_h} \bw_\infty^n \epsilon,\\
&c_3(\epsilon,h) = c_2(\epsilon) + c_4(\epsilon,h),\qquad
&&c_4(\epsilon,h) = h^2 c' c_{\delta_h} \bw_\infty^n  (1+\epsilon^{-1}),
\end{alignat*}
$c'>0$ is a generic constant,
and $\epsilon>0$ is arbitrary.  
The result \eqref{eqn:ExtBound} follows from the inequality $\|\nab \bv\|_{\Omega_h^{n-1}}\le \tnorm{\bv}_{n-1}$
and by
taking $\epsilon$ such that  $c_2(\epsilon)= \frac14$ 
and $h$ sufficiently small such that $c_4(\epsilon,h) \le 1$.
\end{proof}

\begin{lemma}
    There holds the following discrete Poincare inequality
    \begin{equation}\label{eqn:DPoin}
\|\bv\|_{\Omega_h^n}\le c_P \tnorm{\bv}_{n}\qquad \forall \bv\in \bV_h^n.
    \end{equation}
\end{lemma}
\begin{proof}
See \citep[Lemma 7.2]{massing2014stabilized}.
\end{proof}

The following continuity estimate for the 
bilinear form $a^n_h(\cdot,\cdot)$ is essentially
given in \citep{burman2012fictitious} (also see \citep{massing2014stabilized,von2022unfitted}) and follows
from the Cauchy-Schwarz inequality, so the proof is omitted.
\begin{lemma}
There holds
\begin{equation}\label{eqn:anCont}
\begin{aligned}
a^n_h(\bu,\bv)&\lesssim \tnorm{\bu}_{n}\tnorm{\bv}_{n}+ \gamma_s |\bu|_{s_h^n}|\bv|_{s_h^n}\quad &&\forall \bu,\bv\in \bH^{\overline{m}_v+1}(\calT_{h,e}^n)\cap \bH^1(\Omega_{h,e}^n).
\end{aligned}
    \end{equation}
\end{lemma}

The next result states
a discrete trace inequality
for discontinuous piecewise  polynomial functions.  Its proof 
is given in Appendix \ref{app:Dtrace} (also see \citep[Theorem 4.4]{BuffaOrtner09})).
\begin{lemma}\label{lem:Dtrace}
There holds 
\begin{equation}\label{eqn:Dtrace2}
\|q\|_{\Gamma_h^n}\lesssim h^{-1} \qnorme{q}\qquad \forall q\in Q_h^n.
\end{equation}

\end{lemma}

\subsection{Stability analysis}
In this section, we derive
stability results for the finite element method \eqref{eqn:FEM}.
First, we state the energy estimate for the finite element velocity 
in the following lemma. 
This result is essentially given in \citep[Theorem 5.9]{von2022unfitted}, but
we provide a proof for completeness.
\begin{lemma} \label{L1}
    There holds for $h$ sufficiently small,  any $\ep>0$ and $k=1,2,\dots$
    \begin{equation}\label{eqn:DVelStability}
    \begin{split}
    &\|\bu_h^k\|^2_{\Omega^k_h} + \sum_{n=1}^k \|\bu_h^n - \bu_h^{n-1}\|_{\Omega^n_h}^2+
    \Delta t \sum_{n=1}^k \left(\frac14 \vnorm{\bu_h^n}^2+(2\gamma_s-L-\frac12) |\bu_h^n|_{s_h^n}^2+2 \gamma_J |p_h^n|_{J_h^n}^2\right)\\
    &\hspace{0.75in}\le 
    {\rm exp}(ct_k)\left(\|\bu_h^0\|_{\Omega^0_h}^2+\frac{\Delta t}4 \tnorm{\bu^0_h}_{0}^2+ \Delta t L |\bu^0_h|_{s_h^0}^2 \right. \\ 
    & \hspace{1.75in}\left. +\Delta t (c_e+\ep^{-1}) \sum_{n=0}^k \|F^n\|_{*}^2 +\Delta t {\ep} 
    \sum_{n=0}^k\qnorme{p_h^n}^2\right).
    \end{split}
    \end{equation}
with constants $c$ and $c_e$ independent of the discretization parameters.
\end{lemma}
\begin{proof}
Taking \rev{$\bv = \bu_h^{n}$} and $q = p_h^n$
in \eqref{eqn:FEM}, adding the two statements, applying \eqref{eqn:EtaCond},
and using the algebraic identity $(a-b)a = \frac12(a^2-b^2) + \frac12(a-b)^2$ yields
\[
\frac12 \|\bu_h^n\|_{\Omega_h^n}^2 - \frac12 \|\bu_h^{n-1}\|_{\Omega_h^n}^2+ \frac12 \|\bu_h^n -\bu_h^{n-1}\|_{\Omega_h^n}^2
+ \Delta t\left(\frac12\vnorm{\bu_h^n}^2 + \gamma_s |\bu_h^{n}|_{s_h^n}^2+\gamma_J |p_h^n|_{J^n_h}^2\right) \le \Delta t F^n(\bu_h^n,p_h^n)
\]
Using \eqref{eqn:Fnorm}
and the Cauchy-Schwarz inequality we estimate the right-hand side as 
follows
\[
\begin{split}
 F^n(\bu_h^n,p_h^n)&\le\|F^n\|_{*}(\vnorme{\bu_h^n}^2+\qnorme{p_h^n}^2)^\frac{1}{2}
 \le  \|F^n\|_{*}(\vnorme{\bu_h^n}+\qnorme{p_h^n})\\
 &\le \sqrt{c_e/2}\|F^n\|_{*}(\vnorm{\bu_h^n}^2+|\bu_h^n|_{s_h^n}^2)^{\frac12}+\|F^n\|_{*}\qnorme{p_h^n}\\
& \le
 \frac12 \Big(c_e+\epsilon^{-1}\Big)\|F^n\|_{*}^2+  \frac14(\vnorm{\bu_h^n}^2+|\bu_h^n|_{s_h^n}^2)+{\frac\ep2}\qnorme{p_h^n}^2,
 \end{split}
\]
where $c_e\ge 1$ satisfies $\vnorme{\bu_h^n}^2\le \frac{c_e}2 (\vnorm{\bu_h^n}^2+|\bu_h^n|_{s_h^n}^2)$ (cf.~\eqref{eqn:normEquivy}).
This yields
\begin{multline}\label{eqn:EnergyA}
 \|\bu_h^n\|_{\Omega_h^n}^2 -  \|\bu_h^{n-1}\|_{\Omega_h^n}^2+  \|\bu_h^n -\bu_h^{n-1}\|_{\Omega_h^n}^2
+ {\Delta t} \left(\frac12 \tnorm{\bu_h^n}_{n}^2 +(2 \gamma_s-\frac12) |\bu_h^n|_{s_h^n}^2+2 \gamma_J |p_h^n|_{J^n_h}^2\right)\\ \le  {\Delta t}\Big(\big(c_e+\ep^{-1}\big)\|F^n\|_{*}^2+ \ep\qnorme{p_h^n}^2\Big).
\end{multline}
Applying the estimate \eqref{eqn:ExtBound} (with $\bv = \bu_h^{n-1}$)
into \eqref{eqn:EnergyA} and summing the result
 over $n=1,\ldots,k$ yields
\begin{multline*}
\|\bu_h^k\|_{\Omega_h^k}^2 + \sum_{n=1}^k \|\bu_h^n - \bu_h^{n-1}\|_{\Omega_h^n}^2
+ {\Delta t} \sum_{n=1}^k \left(\frac14 \vnorm{\bu_h^n}^2  + (2\gamma_s-L-\frac12)|\bu^n_h|_{s_h^n}^2 + 2 \gamma_J |p_h^n|_{J_h^n}^2\right)\\
\le \|\bu_h^0\|_{\Omega_h^0} + \frac{\Delta t}{4}\tnorm{\bu_h^0}_{0}^2+ \Delta t L |\bu_h^{0}|_{s_h^0}^2+c_1 \Delta t \sum_{n=0}^{k-1} \|\bu_h^n\|_{\Omega_h^n}^2\\
+ \Delta t \sum_{n=1}^k\Big( \big(c_e+\ep^{-1}\big) \|F^n\|_{*}^2+ \ep\qnorme{p_h^n}^2\Big).
\end{multline*}
The estimate \eqref{eqn:DVelStability} now follows from a discrete Gronwall inequality.
\end{proof}


For the complete stability result we need to estimate the pressure term on the right-hand side of \eqref{eqn:DVelStability}. The estimate is given in the next lemma.
\begin{lemma} \label{L2}
    Assume $h^2\lesssim \Delta t$. Then
    \[
    \Delta t \sum_{n=1}^k \qnorme{p_h^n}^2\lesssim {\rm exp}(ct_k)\left(\|\bu_h^0\|_{\Omega^0_h}^2+{\Delta t}(\tnorm{\bu^0_h}_{0}^2+  |\bu^0_h|_{s_h^0}^2) +\Delta t \sum_{n=0}^k \|F^n\|_{*}^2\right).
    \]
\end{lemma}
\begin{proof}
Let $\bv\in \bV_h^n$ satisfy \eqref{eqn:DinfSupExtra} with $q=p_h^n$.
Then using the identity \eqref{eqn:FEM} and  bounds in \eqref{eqn:DinfSupExtra}, \eqref{eqn:Fnorm}, \eqref{eqn:anCont}, 
 and \eqref{eqn:normEquivy}, we have
\begin{align*}
    \qnorm{p_h^n}^2
    &\le b^n_h(\bv,p_h^n)\\
    & = \int_{\Omega^n_h} \frac{\bu_h^n-\bu_h^{n-1}}{\Delta t} \cdot \bv\, dx + a_h^n (\bu^n_h,\bv) - F^n(\bv,0)\\ 
    &\lesssim  \left( \frac1{\Delta t} \|\bu_h^n -\bu_h^{n-1}\|_{\Omega_h^n} \|\bv\|_{\Omega_h^n} + \vnorm{\bu_h^n} \vnorm{\bv} + \gamma_s |\bu_h^n|_{s_h^n} |\bv|_{s_h^n}
    +\|F^n\|_{*} \vnorme{\bv}\right)\\
    &\lesssim \left(\frac{h}{\Delta t} \|\bu_h^n - \bu_h^{n-1}\|_{\Omega^n_h} +\vnorm{\bu_h^n}+\gamma_s \rev{|\bu_h^n|_{s_h^n}} +\|F^n\|_{*} \right)\qnorm{p_h^n}. 
\end{align*}
Thus, we have
\begin{align}\label{eqn:L2Start}
\Delta t  \qnorm{p_h^n}^2& \lesssim \frac{h^2}{\Delta t} \|\bu_h^n - \bu_h^{n-1}\|^2_{\Omega^n_h} +\Delta t (\vnorm{\bu_h^n}^2+ |\bu_h^n|_{s_h^n}^2+\|F^n\|_{*}^2).
\end{align}
Combining this with  \eqref{eqn:normEquivy} leads to the estimate of the pressure norm in the extended domain: 
\begin{align*}
\Delta t  \qnorme{p_h^n}^2& \lesssim \frac{h^2}{\Delta t} \|\bu_h^n - \bu_h^{n-1}\|^2_{\Omega^n_h} +\Delta t (|p_h^n|_{J_h^n}^2+\tnorm{\bu_h^n}_{n}^2+ |\bu_h^n|_{s_h^n}^2+\|F^n\|_{\ast}^2).
\end{align*}
Summing  inequality over $n=1,\ldots,k$, and using $h^2\lesssim \Delta t$ and \eqref{eqn:DVelStability} gets
\begin{align*}
\Delta t \sum_{n=1}^k \qnorme{p_h^n}^2
&\lesssim \sum_{n=1}^k \|\bu_h^n - \bu_h^{n-1}\|^2_{\Omega^n_h} +\Delta t \sum_{n=1}^k (|p_h^n|_{J_h^n}^2+\vnorm{\bu^n_h}^2+|\bu_h^n|_{s_h^n}^2)+\Delta t \sum_{n=1}^k \|F^n\|_{*}^2.
\end{align*}
All terms on the right-hand side of the last inequality are estimated in \eqref{eqn:DVelStability}. 
Thus, by  applying \eqref{eqn:DVelStability} with  $\ep$ sufficiently small but independent of the discretization parameters  proves the lemma.
\end{proof}

\begin{remark}\rm
The corresponding BDF2 scheme is analogous to \eqref{eqn:FEM}, but where the
discrete time derivative is replaced by $\frac{3\bu^{n+1}_h-4\bu_h^{n-1}+\bu_h^{n-2}}{2\Delta t}$,
and the computational mesh is enlarged. In particular, $\delta_h$ is replaced by $2\delta_h$
in the definition of $\calT_{h,e}^n$ so that functions in $\bV_h^{n-2}$ are well defined in $\Omega_h^n$.
In this setting, a stability result holds for the discrete velocity similar
to \eqref{eqn:DVelStability}, but where $\|\bu_h^n-\bu_h^{n-1}\|_{\Omega_h^n}^2$ is replaced
by $\|\bu_h^n-2\bu_h^{n-1}+\bu_h^{n-2}\|_{\Omega_h^n}^2$. The proof of this result
is similar to that of Lemma \ref{L1} but using a different polarization identity, and so we omit the details.

The stability of the discrete pressure solution in the BDF2 scheme
is more subtle and requires a different argument than
 Lemma \ref{L2}. Analogous to \eqref{eqn:L2Start}, there holds
\begin{align*}
\Delta t  \qnorm{p_h^n}^2& \lesssim \frac{h^2}{\Delta t} \|3\bu_h^n - 4\bu_h^{n-1}+\bu_h^{n-2}\|^2_{\Omega^n_h} +\Delta t (\vnorm{\bu_h^n}^2+ |\bu_h^n|_{s_h^n}^2+\|F^n\|_{*}^2),
\end{align*}
and therefore
\begin{align*}
\sum_{n=1}^k \Delta t  \qnorme{p_h^n}^2
& \lesssim \frac{h^2}{\Delta t} \sum_{n=0}^k \|\bu_h^n\|^2_{\Omega^n_h} 
+\Delta t \sum_{k=1}^n (\vnorm{\bu_h^n}^2+ |\bu_h^n|_{s_h^n}^2+\Delta t |p_h^n|_{J_h^n}^2)+ \Delta t \sum_{n=1}^k \|F^n\|_{*}^2\\
& \le  \frac{T h^2}{\Delta t^2} \max_{0\le n\le N} \|\bu_h^n\|^2_{\Omega^n_h} 
+\Delta t \sum_{k=1}^n (\vnorm{\bu_h^n}^2+ |\bu_h^n|_{s_h^n}^2+\Delta t |p_h^n|_{J_h^n}^2)+ \Delta t \sum_{n=1}^k \|F^n\|_{*}^2.
\end{align*}
Thus, for $h\lesssim \Delta t$, the terms in the right-hand side of this 
expression are uniformly bounded, hence obtaining a stability estimate
for the discrete pressure solution.  Note that when combined
with \eqref{cond1}, we have the relation $\Delta t \simeq h$ in 
the case of BDF2.
\end{remark}

\subsection{Consistency}
The consistency of the scheme \eqref{eqn:FEM}
largely follows  the arguments in \citep[Lemma 5.14]{von2022unfitted}.
First, we identify the  extensions of the smooth exact solution $\calE \bu$ and $\calE p$
with $\bu$ and $p$, respectively, 
both of which satisfy \eqref{u_bound_a}. 
Recall that for $\bu$, we consider the divergence-free extension from \eqref{e:extensiont_cont}.
We then set $\bbU^n = \bu^n - \bu_h^n$ and $\bbP^n = p^n-p_h^n$
to denote the errors at $t_n$.
\begin{lemma}\label{lem:consistency}
    There holds for all $(\bv,q)\in \bV^n_h\times Q_h^n$,
        \begin{align*}
        \int_{\Omega_h^n} \frac{\bbU^n-\bbU^{n-1}}{\Delta t} \cdot \bv\, dx +a_h^n(\bbU^n,\bv)-b_h^n(\bv,\bbP^n)+b_h^n(\bbU^n,q)+\gamma_J J_h^n(\bbP^n,q) = \mathfrak{C}^n_c(\bv,q),
    \end{align*}
    where the consistency error $\mathfrak{C}^n_c(\bv,q)$ satisfies
    \begin{align}\label{estCc}
       &|\mathfrak{C}^n_c(\bv,q)|
       \lesssim    h^{q} \|\bu^n\|_{H^{2}(\Omega^n)}\qnorme{q}\\
       &\nonumber\quad+(\Delta t + h^q + h^{m_1}+h^{m_2}) \left(\|{\bf f}^n\|_{H^1(\Omega^n)} 
       + \|\bu\|_{W^{3,5}(\calQ)} + \|p^n\|_{H^{m_2}(\Omega^n)}+ \|\bu^n\|_{H^{m_1+1}(\Omega^n)}\right)\vnorme{\bv},
    \end{align}
for any integers $m_1,m_2$ satisfying $m_1\ge \overline{m}_v$ and $m_2\ge m_q+1$.
\end{lemma}
\begin{proof}
Recall that $\Psi_n:\calO_{\delta_h}(\Omega_h^n)\to \calO_{\delta_h}(\Omega_h)$
is the mapping that connects the approximate
and exact domains and satisfies \eqref{eqn:PhiIsGood}.
Testing \eqref{eqn:Stokes} with $\bv^\ell:=\bv\circ\Psi^{-1}_n$, $\bv\in\bV_h^n$, and $q^\ell:=q\circ\Psi^{-1}_n$, $q\in Q_h^n$, and integrating by parts we arrive at the identity
\begin{multline*}
\int_{\Omega^n} \frac{\partial \bu^n}{\partial t}\bv^\ell\,dx+ \int_{\Omega^n} \nabla\bu^n:\nabla\bv^\ell\,dx -
\int_{\Gamma^n} [(\nabla\bu^n)\bn]\cdot \bv^\ell\,ds
+\int_{\Omega^n} (\bw \cdot \nab \bu^n)  \cdot \bv^\ell)\,dx \\
-\int_{\Omega^n} p^n\Div\bv^\ell\,dx +\int_{\Gamma^n} p^n  (\bv^\ell\cdot \bn)\,ds
-\int_{\Omega^n} q^\ell\Div\bu^n\,dx= \int_{\Omega^n} {\bf f}^n\cdot \bv^\ell\,dx.
\end{multline*}
Subtracting this identity from \eqref{eqn:FEM} gives the  consistency term:
\begin{align*}
\mathfrak{C}^n_c(\bv,q)
&= \underbrace{\int_{\Omega^n} {\bf f}^n\cdot \bv^\ell\,dx -\int_{\Omega^n_h} {\bf f}^n\cdot \bv\,dx}_{=:\mathfrak{T}_1}
+\underbrace{\int_{\Omega_h^n} \frac{{\bu}^n-{\bu}^{n-1}}{\Delta t}\cdot \bv\,dx - \int_{\Omega^n} \frac{\partial \bu^n}{\partial t}\bv^\ell\,dx }_{=:\mathfrak{T}_2}\\
&
+ \underbrace{\widehat a_h^n(\bu^n,\bv) - \int_{\Omega^n} \nabla\bu^n:\nabla\bv^\ell\,dx +
\int_{\Gamma^n} [(\nabla\bu^n)\bn]\cdot \bv^\ell\,ds
-\int_{\Omega^n} (\bw \cdot \nab \bu^n)  \cdot \bv^\ell\,dx }_{=:\mathfrak{T}_3} \\
&
+ \underbrace{\int_{\Omega^n} p^n\Div\bv^\ell\,dx -\int_{\Gamma^n} p^n  (\bv^\ell\cdot \bn)\,ds - b_h^n(\bv,p^n)}_{=:\mathfrak{T}_4}\\
&-\underbrace{b_h^n(\bu^n,q)}_{=:\mathfrak{T}_{5}}
+\underbrace{\gamma_J J_h^n(p^n,q) + \gamma_s s_h^n(\bu^n,\bv)}_{=:\mathfrak{T}_{6}}.
\end{align*}

Estimates for
$\mathfrak{T}_1$ and $\mathfrak{T}_4$ are exactly the same
as in \citep[Lemma 5.14]{von2022unfitted}:
\begin{equation}\label{eqn:mathFrakEst1}
\begin{aligned}
|\mathfrak{T}_1|&\lesssim h^q \|{\bf f}^n\|_{H^1(\Omega^n)} \|\bv\|_{\Omega_h^n},\qquad &&|\mathfrak{T}_4|\lesssim (h^q \|p^n\|_{H^1(\Omega^n)}+ h^{m_2} \|p^n\|_{H^{m_2}(\Omega^n)}) \vnorme{\bv}
\end{aligned}
\end{equation}
for any $m_2\ge 1$.
Likewise, the arguments in \citep[Lemma 5.6]{lehrenfeld2019eulerian}
and \citep[Lemma 5.14]{von2022unfitted} show
\begin{equation}\label{eqn:mathFrakEst1B}
|\mathfrak{T}_2|\lesssim (\Delta t+h^q) \|\bu\|_{W^{2,\infty}(\widehat \Q)} \|\bv\|_{\Omega_h^n}\lesssim 
(\Delta t+h^q)\|\bu\|_{W^{3,5}(\calQ)} \|\bv\|_{\Omega_h^n},
\end{equation}
where we used \eqref{u_bound_a3} in the last inequality.

Unlike the problem considered in \citep{von2022unfitted},
the bilinear form $\widehat a_h^n(\cdot,\cdot)$ includes convective terms.
Nonetheless, the same arguments in \citep[Lemma 5.14]{von2022unfitted}
are valid, yielding the following estimate:
\begin{equation}\label{eqn:mathFrakEst2}
\begin{split}
|\mathfrak{T}_3|
&\lesssim (h^q \|\bu\|_{W^{2,\infty}(\widehat \Q)}+h^{m_1} \|\bu^n\|_{H^{m_1+1}(\Omega^n)})\vnorme{\bv}\\
&\lesssim (h^q \|\bu\|_{W^{3,5}(\calQ)}+h^{m_1} \|\bu^n\|_{H^{m_1+1}(\Omega^n)})\vnorme{\bv},
\end{split}
\end{equation}
where $m_1\ge 1$ is only dictated by the regularity of $\bu^n$,
and we have again used \eqref{u_bound_a3}.

On the other hand, the estimate of $\mathfrak{T}_5=b_h^n(\bu^n,q)$ 
should involve the elementwise  scaled $H^1$-norm for the pressure 
(which is nonstandard and not provided in \citep{von2022unfitted}).
Since the extension of $\bu$ is divergence--free, the estimate of $\mathfrak{T}_5$ 
reduces to estimating the boundary term:
\begin{align*}
\mathfrak{T}_5
&=-\int_{\Gamma_h^n}(\bu^n\cdot\bn)q\, ds. 
\end{align*}
Since $\Psi_n(\Gamma_h^n) = \Gamma^n$, there holds
\[
\bu^n\circ \Psi_n =0\qquad \text{on $\Gamma_h^n$}.
\]
Using the estimate $\|\bu^n-\bu^n\circ \Psi_n\|_{\Gamma_h^n}\lesssim h^{q+1}\|\bu^n\|_{H^2(\Omega^n)}$ 
(cf.~\citep[Lemma 7.3]{gross2015trace}) and the discrete trace inequality in Lemma \ref{lem:Dtrace}, we have
\begin{equation}
\label{eqn:mathFrakEst3}
\begin{split}
|\mathfrak{T}_5|
& = \left|\int_{\Gamma_h^n}(\bu^n-\bu^n\circ\Psi_n)\cdot\bn\, q\, ds\right|
 \le \|\bu^n-\bu^n\circ\Psi_n\|_{\Gamma^n_h} \|q\|_{\Gamma^n_h}
 \lesssim h^{q+1}\|\bu^n\|_{H^{2}(\Omega^n)} \|q\|_{\Gamma^n_h}\\
& \lesssim h^{q}\|\bu^n\|_{H^{2}(\Omega^n)} \qnorme{q}.
\end{split}
\end{equation}

Finally,
the consistency term involving ghost stabilization $\mathfrak{T}_6$ 
vanishes
provided $\bu^n\in \bH^{\overline{m}_v+1}(\Omega_{h,e}^n)$ and $p^n\in H^{m_q+1}(\Omega_{h,e}^n)$.
The estimate \eqref{estCc} then follows from \eqref{eqn:mathFrakEst1}--\eqref{eqn:mathFrakEst3}
and the discrete Poincare inequality \eqref{eqn:DPoin}.
\end{proof}

\subsection{Error Estimates}
In this section, we combine the stability and consistency 
estimates to obtain error estimates for the finite element method \eqref{eqn:FEM}.
As a first step, let $(\bu^n_I,p^n_I)\in \bV_h^n\times Q_h^n$ 
be approximations to the exact solution satisfying 
\begin{subequations}
\label{eqn:upInterp}
\begin{equation}
\begin{split}
    \vnorme{\bu^n-\bu^n_I}+|\bu^n-\bu^n_I|_{s_h^n}
     &\lesssim h^{\underline{m}_v} \|\bu^n\|_{H^{\overline{m}_v+1}(\Omega_{h,e}^n)}\lesssim 
     h^{\underline{m}_v} \|\bu^n\|_{H^{\overline{m}_v+1}(\Omega^n)},\\
     \qnorme{p^n-p^n_I}+|p^n-p^n_I|_{J_h^n}
     &\lesssim h^{{m}_q+1} \|p^n\|_{H^{{m}_q+1}(\Omega_{h,e}^n)}\lesssim 
     h^{{m}_q+1} \|p^n\|_{H^{{m}_q+1}(\Omega^n)},
\end{split}
\end{equation}
and
\begin{equation}
\begin{split}
h^{-1} \|\bu^n-\bu^n_I\|_{\Omega_h^n} &\lesssim  h^{\underline{m}_v} \|\bu\|_{H^{\overline{m}_v+1}(\Omega_{h,e}^n)}\lesssim 
     h^{\underline{m}_v} \|\bu\|_{H^{\overline{m}_v+1}(\Omega^n)},\\
\|p^n-p^n_I\|_{\Omega_h^n}+h^{1/2}\|p^n-p_I^n\|_{\Gamma_h^n}&\lesssim 
h^{{m}_q+1} \|p^n\|_{H^{{m}_q+1}(\Omega_{h,e}^n)}\lesssim 
     h^{{m}_q+1} \|p^n\|_{H^{{m}_q+1}(\Omega^n)}.
     \end{split}
\end{equation}  
\end{subequations}
The existence of such $\bu_I^n$ and $p_I^n$ satisfying 
\eqref{eqn:upInterp} follows from the inclusions 
\eqref{eqn:VhInclusions}--\eqref{eqn:QhInclusions}
and standard scaling and interpolation arguments.
{We also assume the initial condition
of the finite element method \eqref{eqn:FEM} is $\bu_h^0 = \bu_I^0$.}

We then split the error
into its interpolation and discretization errors:
\begin{align*}
        \bbU^n 
        &= \underbrace{(\bu^n - \bu^n_I)}_{=:{\bm \eta^n}} +\underbrace{(\bu_I^n - \bu_h^n)}_{=:\be^n_h\in \bV_h^n},\qquad
        \bbP^n  = \underbrace{(p^n - p^n_I)}_{=:\zeta^n} +\underbrace{(p_I^n - p_h^n)}_{=:{\rm d}_h^n\in Q_h^n}.
\end{align*}
Then the pair $(\be^n_h,{\rm d}_h^n)\in \bV_h^n\times Q_h^n$ satisfies
\begin{align}\label{eqn:errorequation}
    \int_{\Omega_h^n} \frac{\be^n_h-\be_h^{n-1}}{\Delta t}\cdot \bv\, dx +a_h^n(\be_h^n,\bv) -b_h^n(\bv,{\rm d}_h^n) + b_h^n(\be_h^n,q) + \gamma_J J_h({\rm d}_h^n,q) = \mathfrak{C}^n_c(\bv,q)+\mathfrak{C}^n_I(\bv,q), 
\end{align}
for all $(\bv,q)\in \bV_h^n\times Q_h^n$, 
where $\mathfrak{C}^n_c(\bv,q)$ is given in Lemma \ref{lem:consistency}
and 
\begin{align*}
\mathfrak{C}^n_I(\bv,q) 
&= -\underbrace{\int_{\Omega_h^n} \frac{{\bm \eta}^n-{\bm \eta}^{n-1}}{\Delta t}\cdot \bv\, dx }_{=:\mathfrak{T}_7}- \underbrace{a_h^n({\bm \eta}^n,\bv)}_{=:\mathfrak{T}_8} + \underbrace{b_h^n(\bv,\zeta^n)}_{=:\mathfrak{T}_9}-\underbrace{b_h^n({\bm \eta}^n,q)}_{=:\mathfrak{T}_{10}}-\underbrace{\gamma_J J_h^n(\zeta^n,q)}_{=:\mathfrak{T}_{11}}.
\end{align*}
We now bound the terms in $\mathfrak{C}^n_I(\bv,q)$ individually.

First, by continuity estimates and the approximation properties \eqref{eqn:upInterp},
we have
\begin{align}\label{eqn:Test8911}
|\mathfrak{T}_i|\lesssim (h^{\underline{m}_v}\|\bu^n\|_{H^{\overline{m}_v+1}(\Omega^n)}+ h^{m_q+1} \|p^n\|_{H^{m_q+1}(\Omega^n)})\vnorme{\bv}\quad i=8,9,11.
\end{align}
For the temporal interpolation error there holds 
by \citep[Lemma 5.7]{lehrenfeld2019eulerian} and the discrete Poincare inequality \eqref{eqn:DPoin},
\begin{equation}\label{eqn:Test7}
\begin{split}
|\mathfrak{T}_7|
&\lesssim h^{\underline{m}_v} \sup_{t\in [0,T]} \left(\|\bu\|_{H^{\overline{m}_v+1}(\Omega(t))}+ \|\bu_t\|_{H^{\overline{m}_v}(\Omega(t))}\right)\|\bv\|_{\Omega_h^n}\\
&\lesssim h^{\underline{m}_v} \sup_{t\in [0,T]} \left(\|\bu\|_{H^{\overline{m}_v+1}(\Omega(t))}+ \|\bu_t\|_{H^{\overline{m}_v}(\Omega(t))}\right)\vnorme{\bv}.
\end{split}
\end{equation}
For $\mathfrak{T}_{10}$, we 
integrate by parts to obtain 
\[
\mathfrak{T}_{10}
= \int_{\Omega_h^n} (\Div {\bm \eta}^n)q\, dx -\int_{\Gamma_h^n} ({\bm \eta}^n\cdot \bn)q\, ds
=\int_{\Omega_h^n} {\bm \eta}^n \cdot \nabla q\, dx + \sum_{F\in \calF_{h,e}^{n}}\int_{F\cap\Omega_h^n} {\bm \eta}^n\cdot\bn \jump{q}\, ds.
\]
Consequently by an elementwise trace inequality and \eqref{eqn:upInterp}, there holds
 \begin{equation}\label{eqn:Test10}
 \begin{split}
|\mathfrak{T}_{10}|  
 &\le \left(\sum_{T\in \calT_{h,e}^{n}}h_T^{-2}\|{\bm \eta}^n\|_{T}^2\right)^{\frac12}
 \left(\sum_{T\in \calT_{h,e}^{n}}h_T^2 \|\nabla q\|_{T}^2\right)^{\frac12}\\
&\qquad   +\left(\sum_{F\in \calF_{h,e}^n}  h_F^{-1} \|{\bm \eta}^n\|_F^2\right)^{\frac12}
   \left(\sum_{F\in \calF_{h,e}^n}  h_F\left\|\jump{q}\right\|_{F}^2\right)^{\frac12}\\
&
\lesssim \left(\sum_{T\in \calT_{h,e}^{n}}(h_T^{-2}\|{\bm \eta}^n\|_{T}^2+ \|\nabla{\bm \eta}^n\|_{T}^2)\right)^{\frac12}
 \left(\rev{h^2}\sum_{T\in \calT_{h,e}^{n}} \|\nabla q\|_{T}^2+ \rev{h}\sum_{F\in \calF_{h,e}^n}  \left\|\jump{q}\right\|_{F}^2\right)^{\frac12} 
   \\
 &\lesssim  
 h^{\underline{m}_v} \|\bu^n\|_{H^{\overline{m}_v+1}(\Omega^n)} \qnorme{q}.
 \end{split}
\end{equation}
Summarizing \eqref{eqn:Test8911}--\eqref{eqn:Test10}, we proved the bound
\begin{equation}\label{estCI}
\begin{split}
    |\mathfrak{C}_I^n(\bv,q)|\lesssim 
    \Big(h^{\underline{m}_v} \sup_{t\in [0,T]} (\|\bu\|_{H^{\overline{m}_v+1}(\Omega(t))}+ 
    \|\bu_t\|_{H^{\overline{m}_v}(\Omega(t))}) + 
    h^{m_q+1} &\|p^n\|_{H^{m_q+1}(\Omega^n)}\Big) \\
    & \times\big( \vnorme{\bv}+\qnorme{q} \big).
\end{split}
\end{equation}

From \eqref{estCc} and \eqref{estCI}  it follows that the functionals $\mathfrak{C}_c^n$ and $\mathfrak{C}_I^n$ are  bounded
as 
\begin{align*}
|\mathfrak{C}_c^n(\bv,q)|+|\mathfrak{C}_I^n(\bv,q)|
\le C(\Delta t+h^q+h^{\underline{m}_v}+h^{{m}_q}) (\vnorme{\bv}+ \qnorme{q}),
\end{align*}
where $C>0$ depends on  Sobolev norms of the exact solution and the source function.

Note that $\be^n_h$ and $d_h^n$ satisfy the same FE formulation \eqref{eqn:FEM} as $\bu^n_h$ and $p_h^n$ but with the zero initial condition and the right-hand side functional given by 
$F^n(\bv,q)=\mathfrak{C}_c^n(\bv,q)+\mathfrak{C}_I^n(\bv,q)$.
Therefore, we can  apply the stability results from Lemmas~\ref{L1} and~\ref{L2} to estimate $\be^n_h$ and $d_h^n$:
    \begin{equation*}
\|\be_h^k\|^2_{\Omega^k_h} + 
    \Delta t \sum_{n=1}^k \left(\tnorm{\be_h^n}_{n}^2+\qnorme{d_h^n}^2\right)\le C(\Delta t+h^q+h^{\underline{m}_v}+h^{{m}_q+1})^2.
    \end{equation*}

Applying the triangle inequality and the estimates \eqref{eqn:upInterp} one more time leads to our final result. 

\begin{theorem} \label{Th1} Assume the solution $(\bu,p)$ to \eqref{eqn:Stokes} is sufficiently smooth and let $\bu^n_h,p^n_h$ be the solution to \eqref{eqn:FEM}. Assume that the discretization parameters satisfy $h^2\lesssim \Delta t \lesssim h$. The following error estimate holds 
\begin{multline}
\max_{1\le n\le N}\|\bu(t^n)- \bu_h^n\|^2_{\Omega^n_h} + 
    \Delta t \sum_{n=1}^N \left(\tnorm{\bu(t^n)- \bu_h^n}_{n}^2+\qnorme{p(t^n)- p_h^n}^2\right) \\
    \le C(\Delta t+h^q+h^{\underline{m}_v}+h^{{m}_q+1})^2,
\end{multline}
with a constant $C$ independent of discretization parameters but dependent on the solution $(\bu,p)$ and final time $T$.
\end{theorem}

\rev{
\begin{remark}\rm
Compared to the error estimates in \citep{von2022unfitted},
Theorem \ref{Th1} provide optimal-order error estimates
for the velocity and pressure approximations.
The key tool that differentiates our result is the application of 
a scaled $L^2(H^1)$ norm for the pressure approximation.
This strategy provides the flexibility to effectively handle the 
non-divergence-free property 
of the discrete time derivative $(\bu^n-\bu_h^{n-1})/\Delta t$ 
in the stability analysis under the mesh constraint $h^2\lesssim \Delta t \lesssim h$
(cf.~\citep[Lemma 5.10]{von2022unfitted} and Lemma \ref{L2}).
\end{remark}
}

\rev{
\begin{remark}\rm
In the Oseen problem \ref{eqn:Stokes}, we have implicitly
taken the viscosity $\nu=1$ to simplify the presentation.
If $\Delta \bu$ is replaced by $\nu \Delta \bu$,
then the velocity ghost penalty term in CutFEM discretization 
needs to scale like $\nu$ to perform the convergence and stability
analysis. Also, $\gamma_s s_h^n(\cdot,\cdot)$
would be replaced by $\nu^{-1} \gamma_s s_h^n(\cdot,\cdot)$.
In this general setting, a version of Theorem \ref{Th1}
holds, but the constant $C>0$ scales like $\exp(\nu^{-1}T)$; see \citep{von2022unfitted} for details.
\end{remark}
}

\section{Examples of finite element pairs satisfying Assumption \ref{ass:Pair}}\label{sec:Examples}
In this section, we show that several canonical finite element pairs 
for the Stokes problem satisfy the three inequalities \eqref{eqn:DinfSupExtra}
in Assumption \ref{ass:Pair}.

\subsection{The Mini element}
For a tetrahedron $T\in \calT_h$, let $b_T\in \pol_{4}(T)$
denote the standard quartic bubble function, i.e.,
the product of the barycentric coordinates of $T$.
The lowest-order Mini pair with respect to $\calT_h$ is given by \citep{mini84}
\begin{align*}
    \bV_h &= \{\bv\in \bH^1(\whO):\ \bv|_T\in \bpol_1(T)+b_T\bpol_0(T)\ \forall T\in \calT_h\},\\
    Q_h & = \{q\in H^1(\whO):\ q|_T\in \pol_1(T)\ \forall T\in \calT_h\}.
\end{align*}
In this setting we can take $\underline{m}_v = 1$, $\overline{m}_v=4$ and $m_q =1$.

We now verify conditions \eqref{eqn:DinfSupExtra}.
Given $q\in Q_h^n$, we  set $\bv\in \bV_h^n$
so that $\bv|_T = h_T^2 b_T \nab q|_T$ for all $T\in \calT_h^{n,i}$.
The function $\bv$ is extended to $\Omega^n_e$ by zero.
The results in \citep[Section 6.5]{guzman2016inf} show that 
\eqref{eqn:DinfSupExtra1}--\eqref{eqn:DinfSupExtra2} is satisfied.
We also have by a simple scaling argument
\begin{align*}
 \|\bv\|_{\Omega^n}^2 &= \sum_{T\in \calT_h^{n,i}} \|\bv\|_T^2
= \sum_{T\in \calT_h^{n,i}} h_T^4 \|b_T\nab q\|_T^2
\simeq \sum_{T\in \calT_h^{n,i}} h_T^4 \|\nab q\|_T^2
\lesssim h^2 \qnorm{q}^2.
\end{align*}
Thus, \eqref{eqn:DinfSupExtra3} is satisfied as well.

\subsection{The Taylor-Hood pair}
The (generalized) Taylor--Hood finite element pair is given by
\begin{align*}
    \bV_h &= \{\bv\in \bH^1(\whO):\ \bv|_T\in \bpol_m(T)\ \forall T\in \calT_h\},\\
    Q_h & = \{q\in H^1(\whO):\ q|_T\in \pol_{m-1}(T)\ \forall T\in \calT_h\},
\end{align*}
where $m\ge 2$.  Thus, in this case $\overline{m}_v= \underline{m}_v = m$
and $m_q = m-1$ in \eqref{eqn:VhInclusions}--\eqref{eqn:QhInclusions}.
Denote by $\calE_h^{n,i}$ the set of interior one-dimensional edges
of the the interior triangulation $\calT_h^{n,i}$.  We then 
denote by $\tilde \calT_h^{n,i}$ the members in $\calT_h^{n,i}$
that have at least three edges in $\calE^{n,i}_h$ (cf.~Remark \ref{rem:AssumeGeneral}).  We assume that
the domain of pressure ghost-stabilization is chosen such that \eqref{eqn:qExtend} is satisfied.
This is the case provided $c_{\delta_h}$ is sufficiently large (but still $O(1)$).

We denote the set of interior edges of $\tilde \calT_h^{n,i}$
by $\tilde \calE_h^{n,i}$.
Then for $e\in \tilde \calE_h^{n,i}$,
we let $\phi_e$ denote the quadratic bubble function associated with $e$,
and let $\bt_e$ be a unit tangent vector of $e$.
Note that $\phi_e$ has support on the tetrahedra that have $e$ as an edge,
and the number of such tetrahedra is uniformly bounded due to the shape-regularity
of $\tilde \calT_{h,i}^n$.

For a given $q\in Q_h^n$, we define
\[
\bv = \sum_{e\in \tilde \calE_h^{n,i}} h_e^2 \phi_e (\nabla q\cdot \bt_e) \bt_e.
\]
Because $q$ is continuous, we see that $\nabla q\cdot \bt_e$ is single-valued
on $e$, and thus $\bv$ is continuous and a piecewise polynomial of degree
$m$; hence, $\bv\in \bV_h^n$.  

It is shown in \citep[Section 6.1]{guzman2016inf}
that \eqref{eqn:DinfSupExtra1}--\eqref{eqn:DinfSupExtra2} is satisfied,
\rev{thus it remains to show} \eqref{eqn:DinfSupExtra3}.
This follows from the identity
$\|\phi_e\|_{\infty}=1$ and 
the shape-regularity \rev{and quasi-uniformity} of the triangulation:
\begin{align*}
\|\bv\|_{\tilde \Omega_h^n}
\lesssim   \rev{h^2} \sum_{T\in \tilde \calT_{h,i}^{n}}  \|\nab q\|_{T}^2 .
\end{align*}

\subsection{The $\pol_3-\pol_0$ pair}
As our final example, we consider the $\bpol_3-\pol_0$ pair.
In particular,
the discrete velocity space is the cubic Lagrange space,
and the discrete pressure space consists of piecewise constants:
\begin{align*}
    \bV_h & = \bpol_3^c(\calT_h) = \{\bv\in \bH^1(\whO):\ \bv|_T\in \bpol_3(T)\ \forall T\in \calT_h\},\\
    Q_h & = \pol_0(\calT_h) = \{q\in L^2(\whO):\ q|_T\in \pol_0(T)\ \forall T\in \calT_h\}.
\end{align*}
For each interior face $F\in \calF_{h,i}^{n}$ with $F = \p T_1\cap \p T_2$,
we denote by $\bn_j$ the outward unit normal of $\p T_j$ restricted to $F$.
Then for given $q\in Q^n_h$, we define $\bv\in \bV^n_h$ such that for all $F\in \calF_{h,i}^{n}$,
\[
\int_F \bv\cdot \bn_1\, ds
= -h_F \int_F (q_1\bn_1+q_2\bn_2)\cdot \bn_1\, ds = h_F \int_F \jump{q}\cdot \bn_1\, ds,
\]
where $q_j = q|_{T_j}$.  
Note that this condition implies $\int_F \bv\cdot \bn_{F }\, ds = -h_F \int_F\jump{q}\cdot \bn_{F}\, ds$
for any unit normal of $F\in \calF_{h,i}^{n}$.
We further specify that $\bv=0$ on all vertices and edges in $\calT_{h,i}^{n}$,
$\bv\times \bn_F=0$ on all faces $F\in \calF_{h,i}^{n}$, and $\bv=0$ on the boundary of $\Omega_{h,i}^{n}$.
We extend $\bv$ to $\Omega_{h,e}^n$ by zero.

By the divergence theorem, and using that $q$ is piecewise constant,
we have 
\begin{align*}
b_h^n(\bv,q) = \int_{\Omega_{h,i}^n} ({\rm div}\,\bv)q\, dx
&= -\sum_{T\in \calT_{h,i}^{n}} \int_{\p T} q (\bv\cdot \bn_{\p T})\, ds
 \rev{\gtrsim} \rev{h} \sum_{F\in \calF_{h,i}^{n}} \left\|\jump{q}\right\|_F^2 = \qnorm{q}^2.
 \end{align*}
 Thus, \eqref{eqn:DinfSupExtra2} is satisfied.
A scaling argument also yields on each $T\in \calT_{h,i}^{n}$,
\begin{align*}
    |\bv|_{H^m(T)}
    %
    &\lesssim h_T^{2-2m} \sum_{\calF_{h,i}^{n}\ni F\subset \p T} h_F \left\|\jump{q}\right\|_{F}^2.
\end{align*}
Consequently, by another scaling argument,
\begin{align*}
    \vnorme{\bv}^2 &\lesssim \|\nab \bv\|_{\Omega_{h,i}^{n}}^2+h^{-2}\|\bv\|_{\Omega_{h,i}^n}^2
    \lesssim \rev{h} \sum_{F\in \calF_{h,i}^{n}}
     \left\|\jump{q}\right\|_{F}^2 = \qnorm{q}^2,\\
    \|\bv\|_{\Omega_h^{n}}^2 &= \|\bv\|_{\Omega_{h,i}^{n}}^2 \lesssim \rev{h^3} \sum_{F\in \calF_{h,i}^{n,i}}
     \|\jump{q}\|_{F}^2 = h^2 \qnorm{q}^2,
\end{align*}
and therefore \eqref{eqn:DinfSupExtra1} and \eqref{eqn:DinfSupExtra3} are satisfied as well.

\appendix

\section{Proof of Lemma \ref{lem:Dtrace}}\label{app:Dtrace}
We first note that if $Q_h^n\subset H^1(\Omhe)$, then a standard
trace inequality and the definition of $\qnorme{\cdot}$ yields
\begin{equation}\label{eqn:DtraceP}
\|q\|_{\Gamma_h^n} \lesssim \|q\|_{H^1(\Omh)}\lesssim h^{-1}\qnorme{q}+\|q\|_{\Omega_h^n}.
\end{equation}
To establish \eqref{eqn:Dtrace2} in this case, we first apply a standard Poincare-Friedrich inequality
\[
\|q\|_{\Omega_{h,i}^n}\lesssim \|\nab q\|_{\Omega_{h,i}^n}\qquad \forall q\in \mathring{L}^2(\Omega_{h,i}^n)\cap H^1(\Omega_{h,i}^n),
\]
and \eqref{eqn:normEquivy} to conclude
\begin{align*}
    \|q\|_{\Omega_h^n}\lesssim \|q\|_{\Omega_{h,i}^n} +|q|_{J_h^n}\lesssim \|\nab q\|_{\Omega_{h,i}^n}+|q|_{J_h^n}
    \lesssim h^{-1}\big(\qnorm{q}+|q|_{J_h^n}\big)\lesssim h^{-1}\qnorme{q}\quad \forall q\in Q_h^n.
\end{align*}
The estimate \eqref{eqn:Dtrace2} then follows from this inequality and \eqref{eqn:DtraceP}.

Thus, it suffices to prove \eqref{eqn:Dtrace2}
in the case $Q_h^n$ consists of discontinuous polynomials.
To this end, 
we introduce an enriching operator
$E_h:Q_h^n\to Q_h^n\cap H^1(\Omhe)$ constructed
by averaging \citep{BrennerScott3}.
Let
\[
\calT^n_T = \{T'\in \calT_{h,e}^n:\ \bar T\cap \bar T'\neq \emptyset\},
\]
and let $\calF_T^{n,I}$ denote the set of {\em interior} faces
of $\calT^n_T$.  Then there holds
\begin{align}\label{eqn:Enrichy}
|q-E_h q|_{H^\ell(T)}^2 &\lesssim \rev{h_T^{1-2\ell}} \sum_{F\in \calF^{n,I}_T} \left\|\jump{q}\right\|_{L^2(F)}^2\qquad \ell=0,1.
\end{align}

It then follows from \eqref{eqn:Enrichy}
and the trace inequality
\begin{align*}
\|q\|_{T\cap \Gamma_h^n} \lesssim h_T^{-1/2} \|q\|_T + h_T^{1/2} \|\nab q\|_T\qquad \forall q\in H^1(T)
\end{align*}
that
\begin{equation}
\label{eqn:EstqEBndy}
\begin{split}
\|q-E_h q\|_{\Gamma_h}^2 
&= \sum_{T\in \calT_{h,e}^n} \|q-E_h q\|_{T\cap \Gamma_h^n}^2\\
&\lesssim \sum_{T\in \calT_{h,e}^n}  \left(h_T^{-1} \|q-E_h q \|_T^2 + h_T  \|\nab (q-E_h q)\|_T^2\right)\\
%
%
&\lesssim \sum_{F\in \calF_{h,e}^{n}} \left\|\jump{q}\right\|_F^2\lesssim h^{-1}\qnorme{q}^2.
 \end{split}
 \end{equation}
 Furthermore by a standard trace inequality and \eqref{eqn:Enrichy}, we have
\begin{equation}\label{eqn:EstqEBndy2}
\begin{split}
\|E_h q\|_{\Gamma_h^n}^2
&\lesssim \|E_h q\|_{H^1(\Omh)}^2\le \|E_h q\|_{H^1(\Omhe)}^2\\
&\lesssim  \sum_{T\in \calT_{h,e}^n} \|q\|_{H^1(T)}^2 + \rev{h^{-1}}\sum_{F\in \calF^n_{h,e}} \left\|\jump{q}\right\|_F^2\\
&\lesssim h^{-2} \qnorme{q}^2+\|q\|_{\Omhe}^2.
\end{split}
\end{equation}
Combining \eqref{eqn:EstqEBndy}--\eqref{eqn:EstqEBndy2} yields 
\begin{align}
\label{eqn:DtraceP2}
\|q\|_{\Gamma_h^n}\lesssim h^{-1} \qnorme{q}+ \|q\|_{\Omhe}.
\end{align}

Finally, since $q|_{\Omega_{h,i}^n}\in \mathring{L}^2(\Omega_{h,i}^n)$,
we apply the discrete Poincare-Friedrich inequality \citep[Theorem 10.6.12]{BrennerScott3}
\begin{align*}
    \|q\|_{\Omega_{h,i}^n}^2
    &\lesssim \sum_{T\in \calT_{h,i}^n} \|\nab q\|_T^2 + 
    \rev{h^{-1}}\sum_{F\in \calF_{h,i}^n} \left\|\jump{q}\right\|_F^2\lesssim h^{-2} \qnorm{q}^2,
\end{align*}
and \eqref{eqn:normEquivy} to conclude
\begin{align*}
    \|q\|_{\Omega_{h,e}^n}\lesssim \|q\|_{\Omega_{h,i}^n}+|q|_{J_h^n}
    \lesssim h^{-1}\big(\qnorm{q} +|q|_{J_h^n}\big)\lesssim h^{-1}\qnorme{q}.
\end{align*}
Combined with \eqref{eqn:DtraceP2}, we obtain \eqref{eqn:Dtrace2}.

\section*{Acknowledgments}
This work is supported in part
by the NSF through the grants DMS-2309425 (Neilan)
 and DMS-2309197 (Olshanskii).

\bibliographystyle{plain}
\bibliography{literatur}

\end{document}